\documentclass[11pt,reqno]{amsart}
\usepackage{amsfonts}
\usepackage{fancyhdr}
\usepackage{times}
\usepackage[T1]{fontenc}
\usepackage{mathrsfs}
\usepackage{latexsym}
\usepackage[dvips]{graphics}
\usepackage{epsfig}
\usepackage{hyperref, amsmath, amsthm, amsfonts, amscd, flafter,epsf}
\usepackage{amsmath,amsfonts,amsthm,amssymb,amscd}
\usepackage{color}
\usepackage[parfill]{parskip}
\usepackage{mathrsfs} % use command \mathscr
\usepackage{geometry} % See geometry.pdf to learn the layout options. There are lots.
\usepackage{epstopdf}

\usepackage{verbatim}

\allowdisplaybreaks
\newtheorem{thm}{Theorem}%[section]

\newtheorem{lem}[thm]{Lemma}
\newtheorem{thm*}{Theorem}
\newtheorem*{con*}{Conjecture}
\newtheorem*{lem*}{Lemma}

\begin{document}
 \baselineskip=17pt
\hbox{}
\medskip

\title{Note on the number of zeros of $\zeta^{(k)}(s)$}
\author{Fan Ge, Ade Irma Suriajaya}

\email{fange.math@gmail.com}
\address{Department of Mathematics, College of William and Mary, USA}

\email{adeirmasuriajaya@math.kyushu-u.ac.jp}
\address{Faculty of Mathematics, Kyushu University, Japan}

\begin{abstract}
Assuming the Riemann hypothesis, we prove that
$$
N_k(T) = \frac{T}{2\pi}\log \frac{T}{4\pi e} + O_k\left(\frac{\log{T}}{\log\log{T}}\right),
$$
where $N_k(T)$ is the number of zeros of $\zeta^{(k)}(s)$ in the region
$0<\Im s\le T$. We further apply our method and obtain a zero counting formula for the derivative of Selberg zeta functions, improving earlier work of Luo~\cite{Luo}.
\end{abstract}

\maketitle

%%%%%%%%%%%%%%%%%%%%%%%%%%%%%%%%%%%
%%%%%%%%%%%%%%%%%%%%%%%%%%%%%%%%%%%

\section{Introduction}

Let $\zeta(s)$ be the Riemann zeta function, and let 
$$
N(T):=\sum_{\substack{0<\gamma\leq T\\ \beta>0}} 1
$$
be the zero counting function for $\zeta(s)$. Here and throughout, 
$\rho=\beta+i\gamma$ is a generic zero of $\zeta(s)$.
It is known that
\begin{align*}
N(T)=\frac{T}{2\pi}\log \frac{T}{2\pi e} + E_0(T),
\end{align*}
where
\begin{align}\label{eq e0}
E_0(T)=\left\{
	\begin{array}{lll}
	& O(\log{T}), & \textrm{ unconditionally,} \\ 
	\\
	& O\left(\frac{\log{T}}{\log\log{T}}\right), & \textrm{ assuming the Riemann hypothesis (RH).}
	\end{array}\right.
\end{align}
The unconditional bound is known as the Riemann-von~Mangoldt formula (see~\cite[Theorem 9.4]{Tit}), and the conditional bound is due to J.~E.~Littlewood~\cite{Lit}.

There has also been considerable interest in zeros of derivatives of $\zeta(s)$.
Let $\zeta^{(k)}(s)$ be the $k$-th derivative of the Riemann zeta function, and let 
$$
N_k(T):=\sum_{\substack{0<\gamma_k\leq T\\ \beta_k>0}} 1
$$
be the zero counting function for $\zeta^{(k)}(s)$. Here and throughout, 
$\rho_k=\beta_k+i\gamma_k$ is a generic zero of $\zeta^{(k)}(s)$.	
In~\cite{Ber} B. C. Berndt proved that
	\begin{align*}
	N_k(T)=\frac{T}{2\pi}\log \frac{T}{4\pi e} + E_k(T)
	\end{align*}
where 
$$ E_k(T)=O_k(\log{T}). $$
This should be compared to the Riemann-von~Mangoldt formula. In view of~\eqref{eq e0} one may expect to prove that, assuming RH,
	\begin{align}\label{eq ek}
	E_k(T)=O_k\left(\frac{\log{T}}{\log\log{T}}\right)
	\end{align}
for all positive integers $k$.
The first result in this direction is due to H. Akatsuka~\cite{Aka}, who showed that if RH is true then
	\begin{equation*}
	E_1(T)=O\left(\frac{\log{T}}{\sqrt{\log\log{T}}}\right).
	\end{equation*}
Yet this bound is weaker than~\eqref{eq ek}.
The second author~\cite{Sur} extended this estimate to higher derivatives and showed that on RH
	\begin{equation*}
	E_k(T)=O_k\left(\frac{\log{T}}{\sqrt{\log\log{T}}}\right)
	\end{equation*}	
for all positive integers $k$. 

Recently, the first author~\cite{Ge} was able to prove~\eqref{eq ek} for $k=1$, namely,
\begin{align*}
E_1(T)=O\left(\frac{\log{T}}{\log\log{T}}\right).
\end{align*}
 A key ingredient in his proof is an upper bound for the number of zeros of $\zeta'(s)$ close to the critical line, and the idea there has its origin in Y. Zhang's work~\cite{Zha}. However, the method for $k=1$ is not readily applicable for larger $k$. The purpose of this note is to modify the method in~\cite{Ge} and show that the estimate~\eqref{eq ek} holds for all positive integers $k$.	
\begin{thm}\label{thm_nk}
	Assume RH. Then we have
	$$
	N_k(T) = \frac{T}{2\pi}\log \frac{T}{4\pi e} + O_k\left(\frac{\log{T}}{\log\log{T}}\right)
	$$
	as $T\to\infty$.
\end{thm}

We remark that Littlewood's conditional bound 
\begin{align}\label{eq littlewood}
E_0(T)=O\left(\frac{\log
	T}{\log\log{T}}\right)
\end{align}
was first proved in 1924. Later in 1944 A.~Selberg~\cite{Sel1} gave a different proof for this result. In 2007, D.~A.~Goldston and S.~M.~Gonek~\cite{GolGon} showed that we can take the implied constant to be $1/2$. The current best known constant is $1/4$, and this is due to E.~Carneiro, V.~Chandee and M.~B.~Milinovich~\cite{CCM} who proved it using two different methods in 2013. It seems difficult to reduce the size of the bound~\eqref{eq littlewood}, and this suggests that the bounds in Theorem~\ref{thm_nk} might be best possible within current knowledge.

On the other hand, using interesting heuristic arguments D. W. Farmer, S. M. Gonek and C. P. Hughes \cite{FGH} have
conjectured that $E_0(T) = O(\sqrt{\log{T}\log\log{T}})$. This raises the question of what bounds one
should expect for $E_k(T)$. We have the following 
\begin{thm}\label{thm_general}
Assume RH and suppose that
$E_0(T)=O(\Phi(T))$ for some increasing function $\log\log{T} \ll \Phi(T)\ll \log{T}$. Then we have
\begin{equation*}
N_k(T) = \frac{T}{2\pi}\log \frac{T}{4\pi e} + O_k\left(\max\left\{\Phi(2T), \sqrt{\log{T}}\log\log{T} \right\}\right).
\end{equation*}
\end{thm}

Clearly Theorem~\ref{thm_nk} is a consequence of Theorem~\ref{thm_general}, so we shall only prove the latter.
We also remark that our method works well for some other zeta and $L$-functions in the $T$-aspect. In Section~\ref{sec gen} we give a brief discussion on this. In particular, we prove an analogue of Weyl's law for the derivative of Selberg zeta functions.

%\textit{\underline{From this point on, we assume RH. All implied constants may depend on $k$}.}

%%%%%%%%%%%%%%%%%%%%%%%%5
%%%%%%%%%%%%%%%%%%%%%%%%%

\section{Lemmas}

Throughout, let $\Phi(T)$ be an increasing function satisfying $\log\log{T}\ll\Phi(T)\ll\log{T}$ and assume that $E_0(T)\ll\Phi(T)$.
Further, we use the variables $k$ and $\ell$ to denote orders of differentiation, where they are
always positive integers.

We first express the error term of $N_k(T)$ in terms of arguments of certain functions.
\begin{lem}\label{lem N1=arg..}
Assume RH. Let $k\geq2$ be an integer.
For $T\ge 2$ satisfying $\zeta(\sigma+iT)\ne 0$ and
$G_k(\sigma+iT)\ne 0$ for all $\sigma\in\mathbb{R}$, we have
$$
N_k(T) = \frac{T}{2\pi}\log \frac{T}{4\pi e} + \frac{1}{2\pi}\arg G_k(1/2+iT)+\frac{1}{2\pi}\arg \zeta(1/2+iT) + O_k(1),
$$
where $$G_k(s)=\frac{2^s(-1)^k}{(\log 2)^k}\zeta^{(k)}(s),$$ and the argument is
defined by continuous variation from $+\infty$, with the argument at
$+\infty$ being 0.
\end{lem}

\proof This is standard. Apply the argument principle to $\frac{G_k}{G_{k-1}}(s)$
%the function $\frac{G_k}{G_{k-1}}(s)$
on the rectangular region with vertices
$1/4+i,\sigma_k+i,\sigma_k+iT,1/4+iT$, where $\sigma_k$ is large so that $G_k$ is dominated by $1$ to the right of $\sigma_k$.
See also \cite[Proposition 3.1]{Sur} for an alternative proof.
\qed

\begin{lem}\label{lem_arg_g_over_zeta}
Assume RH and let $\ell\geq1$ be an integer.
Then for $1/2+\frac{(\log\log{T})^2}{\log{T}}<\sigma<1$, we have
$$
\arg G_\ell(\sigma+iT) \ll_\ell \Phi(T) + \frac{\log\log{T}}{\sigma-\frac{1}{2}}.
$$
\end{lem}
\proof 
This is follows from \cite[Lemma 2.3]{Sur} by taking $\epsilon_0=(4\log{T})^{-1}$ there.
\qed

\begin{lem}\label{lem_zeta'}
	Let $\ell\geq1$ be an integer.
	For all $t$ sufficiently large we have
	$$
	\frac{G'_\ell}{G_\ell}(s)=\sum_{|\gamma_\ell-t|<1}\frac{1}{s-\rho_\ell}
	+O_\ell(\log{t}),
	$$
	uniformly for $1/2\le \sigma \le 1$.
\end{lem}
\proof This can be proved in a standard way. See Theorem 9.6 (A) in
\cite{Tit} for example. \qed

\begin{lem} \label{log.der.zeta_k}
Assume RH and let $\ell\geq1$ be an integer. Then
$$ \Re \frac{\zeta^{(\ell)}}{\zeta^{(\ell-1)}}(\sigma+it) < 0 $$
holds for $0<\sigma\leq1/2$ and sufficiently large $t$ whenever $\zeta^{(\ell-1)}(\sigma+it)\neq0$.
%Here $\zeta^{(0)} := \zeta$.
\end{lem}

\begin{proof}
Put
$$ \xi_\ell(s) := \Gamma\left(\frac{s}{2}\right) \zeta^{(\ell)}(s). $$

Then we have
$$
\frac{\xi'_{\ell-1}}{\xi_{\ell-1}}(s)
= \frac{1}{2} \frac{\Gamma'}{\Gamma}\left(\frac{s}{2}\right) + \frac{\zeta^{(\ell)}}{\zeta^{(\ell-1)}}(s).
$$
Using Hadamard factorization we easily see that for large $t$,
$$
\frac{1}{2} \frac{\Gamma'}{\Gamma}\left(\frac{s}{2}\right)
= \sum_{n=1}^\infty \left( \frac{1}{2n} - \frac{1}{s+2n} \right) + O(1)
$$
and
$$
\frac{\zeta^{(\ell)}}{\zeta^{(\ell-1)}}(s)
= \sum_{\rho_{\ell-1}} \left( \frac{1}{s-\rho_{\ell-1}} + \frac{1}{\rho_{\ell-1}} \right) + O(1),
$$
where $\rho_{\ell-1}$ runs over all zeros of $\zeta^{(\ell-1)}(s)$.
We can rewrite the latter as
$$
\frac{\zeta^{(\ell)}}{\zeta^{(\ell-1)}}(s)
= \left(\sum_{\beta_{\ell-1}\geq1/2} + \sum_{\substack{\beta_{\ell-1}<1/2,\\ \gamma_{\ell-1}\neq0}} + \sum_{\gamma_{\ell-1}=0}\right) \left( \frac{1}{s-\rho_{\ell-1}} + \frac{1}{\rho_{\ell-1}} \right) + O(1).
$$
By \cite[Corollary of Theorem 7]{LM}, RH implies that $\zeta^{(\ell)}(s)$ has at most finitely many non-real zeros in $\Re (s)<1/2$.
This implies that the second sum is $O(1)$.
Meanwhile \cite{Spi70} shows that
$$
\sum_{\gamma_{\ell-1}=0} \left( \frac{1}{s-\rho_{\ell-1}} + \frac{1}{\rho_{\ell-1}} \right)
= \sum_{j=1}^\infty \left( \frac{1}{s-(-2j+O(1))} + \frac{1}{-2j+O(1)} \right).
$$
Thus
\begin{align*}
\frac{\xi'_{\ell-1}}{\xi_{\ell-1}}(s)
&= \frac{1}{2} \frac{\Gamma'}{\Gamma}\left(\frac{s}{2}\right) + \frac{\zeta^{(\ell)}}{\zeta^{(\ell-1)}}(s) \\
&= \sum_{n=1}^\infty \left( \frac{1}{2n} - \frac{1}{s+2n} \right) + O(1) \\
&\qquad+ \sum_{\beta_{\ell-1}\geq1/2} \left( \frac{1}{s-\rho_{\ell-1}} + \frac{1}{\rho_{\ell-1}} \right)
+ \sum_{j=1}^\infty \left( \frac{1}{s-(-2j+O(1))} + \frac{1}{-2j+O(1)} \right)
+ O(1) \\
&= \sum_{\beta_{\ell-1}\geq1/2} \left( \frac{1}{s-\rho_{\ell-1}} + \frac{1}{\rho_{\ell-1}} \right) + O(1)
\end{align*}
when $t$ is large.

Taking the real part, we have
$$
\Re \frac{\xi'_{\ell-1}}{\xi_{\ell-1}}(s)
%= \sum_{\beta_{\ell-1}\geq1/2} \left( \frac{\sigma-\beta_{\ell-1}}{|s-\rho_{\ell-1}|^2} + \frac{\beta_{\ell-1}}{|\rho_{k-1}|^2} \right) + O(1)
= \sum_{\beta_{\ell-1}\geq1/2} \frac{\sigma-\beta_{\ell-1}}{|s-\rho_{\ell-1}|^2} + O(1).
$$
Hence using Stirling's formula for the Gamma function, we have
\begin{equation} \label{eq:zeta-j_log.der.}
\Re \frac{\zeta^{(\ell)}}{\zeta^{(\ell-1)}}(\sigma+it)
= \Re \frac{\xi'_{\ell-1}}{\xi_{\ell-1}}(\sigma+it) - \frac{1}{2} \Re \frac{\Gamma'}{\Gamma}\left(\frac{\sigma+it}{2}\right)
= \sum_{\beta_{\ell-1}\geq1/2} \frac{\sigma-\beta_{\ell-1}}{|s-\rho_{\ell-1}|^2}
- \frac{1}{2}\log{t} + O(1),
\end{equation}
which is negative for $\sigma\leq1/2$ and $t$ large.
\end{proof}

\begin{lem}\label{lem_line} Assume RH.
	Let $\mathcal{Z}_\ell=\{z_i\}_i$ be the collection of distinct ordinates of zeros of $\zeta, \zeta',...,\zeta^{(\ell)}$ on $\Re (s)=1/2$.
For large $T$ and $Y\leq T$, we have
$$\sum_{T<z\le T+Y,\, z\in \mathcal{Z}_\ell} 1 \ll \Phi(2T)+Y\log{T}. $$
\end{lem}
\proof
It follows from Lemma \ref{log.der.zeta_k} that for all $j\in\mathbb{N}$, zeros of $\zeta^{(j)}$ on the critical line at large heights can only occur at zeros of $\zeta^{(j-1)}$.
Therefore, for sufficiently large $T$,
$$ \sum_{T<z\le T+Y,\, z\in \mathcal{Z}_\ell} 1 \leq \sum_{T<\gamma\le T+Y} 1 \ll \Phi(2T)+Y\log{T}. $$
\qed 

%The following is an immediate consequence of Lemma \ref{lem_line}:
%$$\sum_{T-Y<z\le T+Y,\, z\in \mathcal{Z}_\ell} 1 \ll_\ell \Phi(2T)+Y\log{T}. $$

%%
Write $\mathcal{D}=\mathcal{D}(T)$ for the region $\{w: \Re w\ge 1/2, |\Im w-T| \le 1\}$. Divide $\mathcal{D}$ into $N$ parts, as follows.
Let $B_j=\{w: 1/2\le \Re w\le 1/2+Y_j, |\Im w-T|\leq Y_j\}$ where $Y_j=2^jX$ and $X=(\log{T})^{-1/2}$.
We can write $\mathcal{D}=\cup_{j=1}^N R_j$ where $R_1=B_1$ and $R_j=(B_j-B_{j-1})\cap \mathcal D$ for $j\ge 2$. Note that $2^NX \approx 1$.

 A key ingredient in~\cite{Ge} is an upper bound for the number of zeros of $\zeta'(s)$ in regions like $R_j$'s. To prove our Theorem~\ref{thm_general} we need such bounds for higher derivatives, and the following result provides us the desired estimates.

\begin{lem}\label{lem box}
	Let $N_k(R_j)$ be the number of zeros of $\zeta^{(k)}$ in $R_j$.
	Then $N_k(R_j) \ll_k Y_j\log{T}+\Phi(2T)$.
\end{lem}
\proof
Let $R_j^*$ be $R_j$ without the left side boundary on the critical line. In view of Lemma~\ref{lem_line} it suffices to prove
$N_k(R_j^*) \ll_k Y_j\log{T}+\Phi(2T)$.
Denote by $\Theta(\rho_k;1/2+i(T+Y_j),1/2+i(T-Y_j))\in(0,\pi)$ the argument of the angle at $\rho_k$ with two rays through $1/2+i(T-Y_j)$ and $1/2+i(T+Y_j)$.
Note that $\Theta(\rho_k;1/2+i(T+Y_j),1/2+i(T-Y_j))\gg1$ if $\rho_k \in R_j^*$. Thus
\begin{align}\label{eq R*}
N_k(R_j^*)& \ll \sum_{\rho_k\in R_j^*} \Theta(\rho_k;1/2+i(T+Y_j),1/2+i(T-Y_j)) \notag\\
&= \sum_{\rho_k\in R_j^*}\int_{T-Y_j}^{T+Y_j}\frac{\beta_k-1/2}{(\beta_k-1/2)^2+(\gamma_k-t)^2}dt\notag\\
&= \int_{T-Y_j}^{T+Y_j} \sum_{\rho_k\in R_j^*} \frac{\beta_k-1/2}{(\beta_k-1/2)^2+(\gamma_k-t)^2}dt\notag\\
&\le \int_{T-Y_j}^{T+Y_j}
\sum_{\beta_k>1/2}\frac{\beta_k-1/2}{(\beta_k-1/2)^2+(\gamma_k-t)^2}dt \notag\\
&\le \sum_{\substack{T-Y_j\le z_{i}\le T+Y_j,\\ z_{i}\in \mathcal{Z}_k}} \int_{z_{i}}^{z_{i+1}}
\sum_{\beta_k>1/2}\frac{\beta_k-1/2}{(\beta_k-1/2)^2+(\gamma_k-t)^2}dt.
\end{align}
Write $$ F_k(t)=\sum_{\beta_k>1/2} \frac{\beta_k-1/2}{(\beta_k-1/2)^2+(\gamma_k-t)^2}. $$
Recall \eqref{eq:zeta-j_log.der.} that
$$F_k(t) = -\Re\frac{\zeta^{(k+1)}}{\zeta^{(k)}}(1/2+it)+O(\log t). $$
We claim that
$$
\int_{z_{i}}^{z_{i+1}} F_k(t) dt \ll 1+\log{T}\cdot(z_{i+1}-z_{i}).
$$
To prove this, note that for $t$ on the segment $(z_{i},z_{i+1})$ we can write
$$
\zeta^{(k)}(1/2+it) = \left( h\zeta\cdot \frac{\zeta'}{\zeta}\cdot \frac{\zeta''}{\zeta'}\cdots \frac{\zeta^{(k)}}{\zeta^{(k-1)}}\cdot \frac{1}{h} \right) (1/2+it)
$$
where $h(s)=\pi^{-s/2}\Gamma(s/2)$.
Thus, by using the temporary notation $\Delta\arg$ to denote the argument change along the segment $(z_{i},z_{i+1})$, we have
\begin{align*}
\int_{z_{i}}^{z_{i+1}} F_k(t) dt
&= \int_{z_{i}}^{z_{i+1}} \left(-\Re\frac{\zeta^{(k+1)}}{\zeta^{(k)}}(1/2+it)+O(\log t)\right)dt \\
&= \left|\Delta\arg\zeta^{(k)}(1/2+it)\right|+O\left((z_{i+1}-z_{i})\log{t}\right) \\
&= \Bigg| \Delta\arg(h(1/2+it)\zeta(1/2+it))+\Delta\arg\frac{\zeta'}{\zeta}(1/2+it)+\cdots+\Delta\arg\frac{\zeta^{(k)}}{\zeta^{(k-1)}}(1/2+it) \\
&\qquad+ \Delta\arg\frac{1}{h(1/2+it)} \Bigg| + \left((z_{i+1}-z_{i})\log{t}\right) \\
&\le \left| \Delta\arg(h(1/2+it)\zeta(1/2+it)) \right|
+ \sum_{l=1}^k \left| \Delta\arg\frac{\zeta^{(l)}}{\zeta^{(l-1)}}(1/2+it) \right| \\
&\qquad+ \left| \Delta\arg\frac{1}{h(1/2+it)} \right| + \left((z_{i+1}-z_{i})\log{t}\right).
\end{align*}
It follows from the well-known functional equation for $h(s)\zeta(s)$ that 
$$
\Delta\arg(h(1/2+it)\zeta(1/2+it))=0.
$$
From Lemma \ref{log.der.zeta_k}, we have
$$ \Delta\arg\frac{\zeta^{(l)}}{\zeta^{(l-1)}}(1/2+it) \ll 1 $$
for $l=1,2,\cdots,k$ and $t\in(z_{i},z_{i+1})$.
Moreover, by Stirling's formula we obtain
\begin{align*}
 \left| \Delta\arg\frac{1}{h(1/2+it)} \right| \ll \int_{z_j}^{z_{j+1}}\left|\frac{h'}{h}(1/2+it)\right|dt\\
 \ll (z_{j+1}-z_j)\log{T}.
\end{align*}

Thus $$ \int_{z_{i}}^{z_{i+1}} F_k(t) dt \ll_k 1 + (z_{i+1}-z_{i})\log{T}, $$
as claimed.

It then follows from~\eqref{eq R*} and Lemma~\ref{lem_line} that
\begin{align*}
N_k(R_j^*)
&\ll_k \sum_{\substack{T-Y_j\le z_{i}\le T+Y_j,\\ z_{i}\in \mathcal{Z}_k}} \left( 1 + (z_{i+1}-z_{i})\log{T} \right) \\
&\ll \sum_{\substack{T-Y_j\le z_{i}\le T+Y_j,\\ z_{i}\in \mathcal{Z}_k}} 1
+ (\log{T}) \cdot \sum_{\substack{T-Y_j\le z_{i}\le T+Y_j,\\ z_{i}\in \mathcal{Z}_k}} (z_{i+1}-z_{i}) \\
&\ll Y_j\log{T}+\Phi(2T).
\end{align*}
\qed

%%%%%%%%%%%%%%%%%%%%%%%%%%%%%%%%%%%%%%
%%%%%%%%%%%%%%%%%%%%%%%%%%%%%%%%%%%%%%

\section{Proof of Theorem~\ref{thm_general}}

Applying Lemma \ref{lem N1=arg..}, we only need to show that
$$
\arg G_k (1/2+iT) \ll_k \Phi(2T) + \sqrt{\log{T}}\log\log{T}
$$
holds for all $k\in\mathbb{N}$.

Let $X=1/\sqrt{\log{T}}$ as defined in the paragraph preceding Lemma \ref{lem box}.
From Lemma \ref{lem_arg_g_over_zeta}, we see that
$$
\arg G_k(1/2+X+iT) \ll_k \Phi(T)+ \sqrt{\log{T}}\log\log{T}.
$$
It remains to show
\begin{align}\label{eq Delta}
	\Delta:= \arg G_k(1/2+iT)-\arg G_k(1/2+X+iT)\ll\Phi(2T)+\sqrt{\log{T}}\log\log{T}.
\end{align}

From Lemma \ref{lem_zeta'}, we have
\begin{align}\label{eq Delta 2}
|\Delta|&= \left| \Im \int_{1/2}^{1/2+X}\frac{G'_k}{G_k}(\sigma+iT)d\sigma \right| \notag \\
&\ll \sum_{|\gamma_k-T|<1}\Theta(\rho_k;1/2+iT,1/2+X+iT)+X\log{T},
\end{align}
where $\Theta(a;b,c)$ is the (positive) angle at $a$ in the triangle $abc$. 
Hence, it suffices to show that
$$
\sum_{|\gamma_k-T|<1}\Theta(\rho_k;1/2+iT,1/2+X+iT)\ll \Phi(2T)+\sqrt{\log{T}}\log\log{T}.
$$

From \cite[Corollary of Theorem 7]{LM}, RH implies that for sufficiently large $T$, $\zeta^{(k)}$ has no zeros in the left half of the critical strip above $T-1$. Hence we may assume that
$$
\sum_{|\gamma_k-T|<1} \Theta(\rho_k;1/2+iT,1/2+X+iT) = \sum_{\rho_k\in \mathcal{D}} %f(\rho_k)
\Theta(\rho_k;1/2+iT,1/2+X+iT)
$$
where %we write $f(\rho_k)=f_{X,T}(\rho_k)=\Theta(\rho_k;1/2+iT,1/2+X+iT)$ and 
$\mathcal{D}$ is the region defined in the paragraph preceding Lemma \ref{lem box}.
Using the expression $\mathcal{D}=\cup_{j=1}^N R_j$ and Lemma~\ref{lem box} we have
\begin{align*}
	\sum_{\rho_k\in \mathcal{D}} \Theta(\rho_k;1/2+iT,1/2+X+iT)
	&=\sum_{j=1}^{N} \sum_{\rho_k\in R_j} \Theta(\rho_k;1/2+iT,1/2+X+iT) \\
	&\ll \sum_{j=1}^{N} N_k(R_j) \frac{X}{Y_j}\\
	&\ll_k \sum_{j=1}^{N} (2^j X \log{T}+\Phi(2T)) \frac{1}{2^j} \\
	&\ll \Phi(2T) + XN\log{T}.
	\end{align*}
Recall that $X=\frac{1}{\sqrt{\log{T}}}$ and $N\ll_k \log(1/X)\ll_k \log\log{T}$. Thus the above bound is
	$$
	\ll_k \Phi(2T) + \sqrt{\log{T}}\log\log{T}
	$$
	as desired. \qed
	
\section{Other zeta and $L$-functions}\label{sec gen}
Our method works well for some other zeta and $L$-functions in the $T$-aspect. Below we give two examples of the first derivative of Selberg zeta functions and Dirichlet $L$-functions, respectively. Dealing with higher derivatives of these functions would require information about the ``trivial'' zeros of those derivatives, which is not the purpose of this paper.

First, let us consider the Selberg zeta functions on cocompact hyperbolic surfaces. Precisely, let $X$ be a compact Riemann surface of genus $g \ge 2$, and let $Z_X(s)$ be the associated Selberg zeta function. Denote by $\mathcal N(T)$ and $\mathcal N_1(T)$ the zero counting functions for $Z_X(s)$ and $Z_X'(s)$, respectively; so $\mathcal N(T)$ is the number of nontrivial zeros of $Z_X(s)$ up to height $T$, and similarly for $\mathcal N_1(T)$.
Weyl's law tells us that 
$$
\mathcal N(T)=C_X T^2+O\left(\frac{T}{\log{T}}\right)
$$
where $C_X$ is a specific constant depending on $X$.
In~\cite{Luo} W. Luo proved that
$$
\mathcal N_1(T)=C_X T^2+O(T).
$$
Following our method \textit{in an identical manner}, we can prove that
\begin{align}\label{eq sel zeta}
\mathcal N_1(T)=C_X T^2+D_X T+O\left(\frac{T}{\log{T}}\right),
\end{align}
where $D_X$ is a specific constant depending on $X$. Thus~\eqref{eq sel zeta} improves Luo's result.
(Precisely, $C_X=g-1$ and $D_X=-\log N(P_{00})/2\pi$ where $N(P_{00})=\min_{P_0}N(P_{0})$; see page 1143 in~\cite{Luo} for more explanation of the notation.)
The estimate~\eqref{eq sel zeta} was proved by the first author (unpublished) using a different method, but our method here is simpler. 

As another example, let $L(s,\chi)$ be the Dirichlet $L$-function where $\chi$ is a primitive Dirichlet character to the modulus $q$. Let $N(T,\chi)$ be the number of nontrivial zeros of $L(s,\chi)$ with heights between $-T$ and $T$. Define $N_1(T,\chi)$ similarly as the zero counting function for $L'(s,\chi)$. It follows from Selberg's work~\cite{Sel} that on the generalized Riemann hypothesis (GRH)
\begin{align}\label{eq Sel}
N(T,\chi)=\frac{T}{\pi}\log \frac{qT}{2\pi e} + O\left( \frac{\log{qT}}{\log\log{qT}} \right).
\end{align}
As for $L'(s,\chi)$, recently the first author~\cite{Ge2} proved that on GRH we have
\begin{align}\label{eq Ge}
N_1(T,\chi) = \frac{T}{\pi}\log \frac{qT}{2m\pi e} + O\left(\frac{\log{qT}}{\log\log{qT}} + \sqrt{m\log 2m\log{qT}}\right),
\end{align}
where $m$ is the smallest prime number not dividing $q$. This improves earlier work of the second author~\cite{Sur2}. Our method here should give analogues of~\eqref{eq Ge} for higher derivatives of $L(s,\chi)$ once some standard information on trivial zeros of these derivatives is gathered. As a final remark, we note that in the $T$-aspect the error term in~\eqref{eq Ge} is as good as that in~\eqref{eq Sel}. However, in the $q$-aspect $\sqrt{m\log 2m\log{qT}}$ might sometimes be larger than $\frac{\log{qT}}{\log\log{qT}}$. In fact, simple calculation shows that 
$$ \sqrt{m\log 2m\log{qT}}\ll \frac{\log{qT}}{\log\log{qT}} $$
when $m$ is no greater than $\log{qT}/(\log\log{qT})^3$, while the largest possible value for $m$ is about $\log q$. It would be of interest to see if one can remove the second term in the error of~\eqref{eq Ge}.

%%%%%%%%%%%%%%%%%%%%%%%%%%%%%%%%%%%%%%%%%%%%% 

\section*{Acknowledgement}
This project started when the first author was a departmental postdoc fellow at the University of Waterloo
and the second author was a member of iTHEMS under RIKEN Special Postdoctoral Researcher program.
The second author is supported by JSPS KAKENHI Grant Number 18K13400.
%%%%%%%%%%%%%%%%%%%%%%%%%%%%%%%%%%%%%%%%%%%%% 

\end{document}